\documentclass[11pt]{article}
\usepackage{amsfonts,amssymb,amsmath,amscd,latexsym,makeidx,amsthm, color,graphicx,multicol,pdfsync}
\usepackage[utf8]{inputenc}
\usepackage[english]{babel}
\usepackage{comment}
\newtheorem{teo}{Theorem}[section]

\newtheorem{cor}{Corollary}[section]

\newtheorem{prop}{Proposition}[section]

\newcommand{\R}{\mathbb{R}}

\newcommand{\ve}{\varepsilon}

\newcommand{\wspi}{\tilde W^{s,p}_0(I)}
\newcommand{\wsp}{W^{s,p}(\R)}

\title{A note on the Moser-Trudinger inequality in Sobolev-Slobodeckij spaces in dimension one}
\author{Stefano Iula\thanks{The author is supported by the Swiss National Science Foundation project nr. PP00P2-144669.}\\ {\small Universit\"at Basel}\\ {\small \texttt{stefano.iula@unibas.ch} }}

\date{}
\begin{document}
\maketitle

\begin{abstract}
We discuss some recent results by Parini and Ruf on a Moser-Trudinger type inequality in the setting of Sobolev-Slobodeckij spaces in dimension one. We push further their analysis considering the inequality on the whole $\R$ and we give an answer to one of their open questions.
\end{abstract}

\section{Introduction}

A classical result in analysis states that, if $\Omega\subset\R^n$ is an open set with finite measure $|\Omega|$ and Lipschitz boundary, $k$ is a positive integer with $k<n$, and $p\in[1,\frac kn)$, then the Sobolev space $W^{k,p}_0(\Omega)$ embeds continuously in $L^{\frac{np}{n-kp}}(\Omega)$. This results doesn't hold for the critical case $p=\frac nk$, that is $W^{k,\frac nk}_0(\Omega)$ doesn't embed in $L^\infty(\Omega)$. On the other hand Trudinger \cite{Tru}, Pohozaev \cite{Poho}, Yudovich \cite{Yud} and others found that, at least in the case $k=1$, functions in $W^{1,n}_0(\Omega)$ enjoy summability of exponential type. Namely 
$$ W^{1,n}_0(\Omega) \subset \left\{u\in L^1(\Omega)\colon \int_\Omega e^{\beta |u|^{\frac{n}{n-1}}}\, dx<+\infty\right\}$$
for any $\beta<+\infty$. Moser \cite{Mos} sharpened this embedding and determined the optimal exponent $\alpha_n$ such that
\begin{equation}\label{eq:moser}
\sup_{u\in W^{1,n}_0(\Omega),\|\nabla u\|_{L^n(\Omega)}\leq 1}\int_\Omega e^{\alpha_n|u|^{\frac{n}{n-1}}}\, dx<C|\Omega|, \quad\alpha_n:=n\omega_{n-1}^{\frac{1}{n-1}}.
\end{equation}
Here, $\omega_{n-1}$ is the volume of the unit sphere in $\R^n$. In particular the exponent $\alpha_n$ is sharp in the sense that 

\begin{equation}\notag
\sup_{u\in W^{1,n}_0(\Omega),\|\nabla u\|_{L^n(\Omega)}\leq 1}\int_\Omega e^{\alpha|u|^{\frac{n}{n-1}}}\, dx=+\infty
\end{equation}
for any $\alpha>\alpha_n$. Moreover, the supremum in \eqref{eq:moser} becomes infinite as soon as we slightly modify the integrand, namely
\begin{equation}\label{eq:sharpf}
\sup_{u\in W^{1,n}_0(\Omega),\|\nabla u\|_{L^n(\Omega)}\leq 1}\int_\Omega f(|u|)e^{\alpha_n|u|^{\frac{n}{n-1}}}\, dx=+\infty
\end{equation}
for any measurable function $f\colon\R^+\to\R^+$ such that $\lim_{t\to+\infty}f(t)=\infty$. This can be proved, for instance, using the same test functions defined in \cite{Mos}. In \cite{adams} Adams, exploiting Riesz potentials, extended Moser's result to higher order Sobolev spaces $W^{k,p}_0(\Omega)$, $k>1$, $p=\frac{n}{k}$. 

\medskip

In the present work, we are interested in generalizations of \eqref{eq:moser} that concern Sobolev spaces of fractional orders. The usual approach is to consider Bessel potential spaces $H^{s,p}$. In this setting, sharp versions of \eqref{eq:moser} are proven both in the cases of bounded and unbounded domains of $\R^n$, $n\geq 1$ (see \cite{IMM}, \cite{marfrac} and \cite{hyder}).

\medskip

Here, we focus our attention on the case (in general different from the one of Bessel potential spaces) of Sobolev Slobodeckij spaces (see definitions below), which has been recently proposed, together with some open questions, by Parini and Ruf. In \cite{Par-Ruf} they considered $\Omega\subset\R^n$ to be a bounded and open domain, $n\geq 2$ and $sp=n$ and they were able to prove the existence of $\alpha_*>0$ such that the corresponding version of inequality \eqref{eq:moser} is satisfied for any $\alpha\in(0,\alpha_*)$ (see also \cite{Peetre}). Even though the result is not sharp, in the sense that the value of the optimal exponent is not yet known, an explicit upper bound for the optimal exponent $\alpha^*$ is given.

\medskip

As a first step, we extend the results in \cite{Par-Ruf} to the case $n=1$. For any $s\in(0,1)$ and $p>1$, the Sobolev-Slobodeckij space $\wsp$ is defined as 
$$ \wsp:=\left\{u\in L^p(\R)\colon [u]_{\wsp}<+\infty\right\}$$
where $[u]_{\wsp}$ is the Gagliardo seminorm defined by
\begin{equation}\label{eq:gagliardo}
[u]_{\wsp}:=\left(\int_{\R}\int_{\R}\frac{|u(x)-u(y)|^p}{|x-y|^2}\, dx\, dy\right)^{\frac 1p}.
\end{equation}
We will often write $[\cdot]:=[\cdot]_{\wsp}$. The space $\wsp$ is a Banach space with respect to the norm
\begin{equation}\label{eq:norm}
||u||_{\wsp}:=\left(||u||^p_{L^p(\R)}+[u]^p_{\wsp}\right)^{\frac 1p}.
\end{equation}
Let $I$ be an open interval in $\R$. We define the space $\wspi$ as the closure of $\left(C^\infty_0(I),\|u\|_{\wsp}\right)$. An equivalent definition for $\wspi$ can be obtained taking the completion of $C^\infty_0(I)$ with respect to the seminorm $[u]_{\wsp}$ (see \cite[Remark 2.5]{Brasco-Lindgren-Parini}).

\medskip

With a mild adaptation of the techniques used in \cite{Par-Ruf}, we are able to prove that their result holds also in dimension one.

\begin{teo}\label{inequality}
Let $s\in(0,1)$ and $p>1$ be such that $sp=1$. There exists $\alpha_*=\alpha_*(s)>0$ such that for all $\alpha\in[0,\alpha_*)$ it holds

\begin{equation}\label{eq:MT}
\sup_{u\in\wspi, [u]_{\wsp}\leq 1} \int_I e^{\alpha |u|^{\frac{1}{1-s}}}\, dx<\infty.
\end{equation}
Moreover, there exists $\alpha^*=\alpha^*(s):=\gamma_s^{\frac{s}{1-s}}$ such that the supremum in \eqref{eq:MT} is infinite for any $\alpha\in(\alpha^*,+\infty)$.
\end{teo}
It is worth to remark that, as already pointed out in \cite{Par-Ruf}, the exponent $\alpha^*(\frac 12)$ is equal to $2\pi^2$ and it coincides, up to a normalization constant, with the optimal exponent $\pi$ determined in \cite{IMM} in the setting of Bessel potential spaces. 

\medskip

We move now to the case $I=\R$, pushing further the analysis of \cite{Par-Ruf}. An inequality of the form \eqref{eq:MT} cannot hold if we don't consider the full $\wsp$-norm, i.e. we take into account also the term $\|u\|_{L^p(\R)}$. This has been done by Ruf \cite{Ruf-H1} in the case of $H^{1,2}(\R^2)$, see also \cite{IMM}, \cite{hyder} for the case of Bessel potential spaces. We define
\begin{equation}\label{varphi}
\Phi(t):=e^t-\sum_{k=0}^{\lceil p-2\rceil}\frac{t^k}{k!},
\end{equation}
where $\lceil p-2\rceil$ is the smallest integer greater than, or equal to $p-2$.
\begin{teo}\label{inequalityR}
Let $s\in(0,1)$ and $p>1$ be such that $sp=1$. There exists $\alpha_*=\alpha_*(s)>0$ such that for all $\alpha\in[0,\alpha_*)$ it holds

\begin{equation}\label{eq:MTR}
\sup_{u\in\wsp, ||u||_{\wsp}\leq 1} \int_\R \Phi(\alpha|u|^{\frac{1}{1-s}})\, dx<\infty.
\end{equation}

Moreover the supremum in \eqref{eq:MT} is infinite for any $\alpha\in(\alpha^*,+\infty)$, where $\alpha^*$ is as in Theorem \ref{inequality}
\end{teo}

As we shall see, Theorem \ref{inequality} and \ref{inequalityR} are sharp in the sense of \eqref{eq:sharpf}. Indeed  one of the open questions in \cite{Par-Ruf} was whether an inequality of the type 
$$ \sup_{u\in\wspi,[u]_{\wspi}\leq1}\int_I f(|u|)e^{\alpha|u|^{\frac{1}{1-s}}}\, dx<+\infty,$$
where $f\colon\R^+\to\R^+$ is such that $f(t)\to\infty$ as $t\to\infty$ holds true for the same exponents of the standard Moser-Trudinger inequality (see \cite{hyder},\cite{IMM}). For $n=1$ we prove the following

\begin{teo}\label{sharpinterval}
Let $I\subset\R$ be a bounded interval, $s\in(0,1)$ and $p>1$ such that $sp=1$. We have
\begin{equation}\label{eq:MTf}
\sup_{u\in\wspi, [u]_{\wsp}\leq 1} \int_I f(|u|)e^{\alpha^* |u|^{\frac{1}{1-s}}}\, dx=\infty,
\end{equation}
\begin{equation}\label{eq:MTfr}
\sup_{u\in\wsp, \|u\|_{\wsp}\leq 1} \int_\R f(|u|)\Phi(\alpha^* |u|^{\frac{1}{1-s}})\, dx=\infty,
\end{equation}
where $f\colon[0,\infty)\to[0,\infty)$ is any Borel measurable function such that $\lim_{t\to+\infty}f(t)=\infty$.
\end{teo}

\noindent\textbf{Acknowledgements} I am thankful to my advisor Prof. Luca Martinazzi for suggesting me this problem and for many interesting discussions.

\section{Proof of Theorem \ref{inequality}}

We start this section proving the validity of the Moser-Trudinger inequality \eqref{eq:MT}. The result for $n\geq 2$ is proved in \cite{Par-Ruf} and the proof in the one dimensional case, which we report here for the sake of completeness, follows by a mild adaptation of the techniques in \cite{Par-Ruf}.

\medskip 

Thanks to \cite[Theorem $9.1$]{Peetre}, using Sobolev embeddings and H\"older's inequality we have that there exists a constant $C>0$ independent of $u$ such that for any $u\in\wspi$
\begin{equation}\label{eq:embed}
||u||_{L^q(\R)}\leq C[u]_{\wsp}q^{1-s}
\end{equation}
for any $q>1$. For $[u]_{\wsp}\leq 1$ we write
\begin{equation}\label{eq:usostirling}
\int_I e^{\alpha |u|^{\frac{1}{1-s}}}\, dx=\sum_{k=0}^\infty\int_I \frac{\alpha^k}{k!}|u|^{\frac{k}{1-s}}\, dx\leq \sum_{k=0}^\infty\frac{1}{k!}\left(\frac{C}{1-s}\alpha k\right)^k,
\end{equation}
where in the last inequality we used \eqref{eq:embed}. Thanks to Stirling's formula
\begin{equation}\label{eq:stirling}
k!=\sqrt{2\pi k}\left(\frac ke\right)^k\left(1+O(\frac 1k)\right)
\end{equation}
the series in \eqref{eq:usostirling} converges for small $\alpha$ and we recover a bound (uniform w.r.t. $u$) for 
\begin{equation}\notag
\int_I e^{\alpha |u|^{\frac{1}{1-s}}}\, dx,
\end{equation}
yielding \eqref{eq:MT}.

\medskip

As a direct consequence of \eqref{eq:MT}, using the density of $C^\infty_c(I)$ in $\wspi$, we have the following corollary (see \cite[Proposition 3.2]{Par-Ruf}).

\begin{cor}\label{prop2}
If $u\in\wspi$, for every $\alpha>0$ it holds
\begin{equation}\notag
\int_I e^{\alpha |u|^{\frac{1}{1-s}}}\, dx<\infty.
\end{equation}
\end{cor}

We now give a useful result on the Gagliardo seminorm of radially symmetric functions (see \cite[Proposition 4.3]{Par-Ruf}), which will turn out to be useful later on.

\begin{prop}\label{prop:norm radial}
Let $u\in\wsp$ be radially symmetric and let $sp=1$. Then

\begin{equation}\label{eq:norm radial}
[u]_{\wsp}=\int_\R\int_\R \frac{|u(x)-u(y)|^p}{|x-y|^2}\, dx\, dy=4\int_0^{+\infty}\int_0^{+\infty}|u(x)-u(y)|^p\frac{x^2+y^2}{(x^2-y^2)^2}\, dx\, dy
\end{equation}
\end{prop}
\begin{proof}
The proof will follow from a direct computation. We split

\begin{equation}\notag
\begin{split}
&\int_\R\int_\R \frac{|u(x)-u(y)|^p}{|x-y|^2}\, dx\, dy\\
=&\int_0^{+\infty}\int_0^{+\infty} \frac{|u(x)-u(y)|^p}{|x-y|^2} \, dx\, dy+\int_{-\infty}^0 \int_{-\infty}^0  \frac{|u(x)-u(y)|^p}{|x-y|^2}\, dx\, dy \\
+&\int_0^{+\infty}\int_{-\infty}^0  \frac{|u(x)-u(y)|^p}{|x-y|^2}\, dx\, dy+ \int_{-\infty}^0\int_0^{+\infty}  \frac{|u(x)-u(y)|^p}{|x-y|^2}\, dx\, dy.
\end{split}
\end{equation}
Using a straightforward change of variable and the symmetry of $u$, we obtain the claim.
\end{proof}

\medskip

To give an upper bound for the optimal exponent $\bar \alpha$ such that the supremum in \eqref{eq:MT} is finite for $\alpha\in[0,\bar\alpha)$, we define the family of functions

\begin{equation}\label{eq:moserfunctions}
u_\ve(x):=
\begin{cases}
|\log\ve|^{1-s}\quad&\text{ if }|x|\leq\ve\\
\frac{|\log|x||}{|\log\ve|^{s}}\quad&\text{ if }\ve<|x|<1\\
0\quad&\text{ if }|x|\geq 1.
\end{cases}
\end{equation}
Notice that the restrictions of $u_\ve$ to $I$ belong to $\wspi$. 

\begin{prop}\label{const}
Let $sp=1$ and $(u_\ve)\subset\wspi$ be the family of functions defined in \eqref{eq:moserfunctions}. Then
\begin{equation}\label{eq:limsemi}
\lim_{\ve\to0}[u_\ve]_{\wsp}^p=\gamma_s:=8\, \Gamma(p+1)\sum_{k=0}^\infty \frac{1}{(1+2k)^p}.
\end{equation}
\end{prop}

\begin{proof}
We will follow the proof in \cite{Par-Ruf}. Define

\begin{equation}\label{eq:ieps}
I(\ve):=\int_\R\int_\R \frac{|u_\ve(x)-u_\ve(y)|^p}{|x-y|^2}\, dx\, dy.
\end{equation}
Using Proposition \ref{prop:norm radial} and  \eqref{eq:moserfunctions} we see that $I(\ve)$ can be decomposed as 
$$I(\ve)=I_1(\ve)+I_2(\ve)+I_3(\ve)+I_4(\ve),$$ 

where
\begin{equation}\notag
I_1(\ve)=\frac{8}{|\log\ve|}\int_\ve^1\int_0^\ve |\log x-\log\ve|^p \frac{x^2+y^2}{(x^2-y^2)^2}\, dx\, dy,
\end{equation}

\begin{equation}\notag
I_2(\ve)=\frac{4}{|\log\ve|}\int_\ve^1\int_\ve^1 |\log x-\log y|^p \frac{x^2+y^2}{(x^2-y^2)^2}\, dx\, dy,
\end{equation}

\begin{equation}\notag
I_3(\ve)=8|\log\ve|^{p-1}\int_1^{+\infty}\int_0^\ve \frac{x^2+y^2}{(x^2-y^2)^2}\, dx\, dy,
\end{equation}

\begin{equation}\notag
I_4(\ve)=\frac{8}{|\log\ve|}\int_\ve^1\int_1^{+\infty} |\log x |^p \frac{x^2+y^2}{(x^2-y^2)^2}\, dx\, dy.
\end{equation}

With an integration by parts, it is easy to check that $\lim_{\ve\to0}I_i(\ve)=0$ for $i=1,3,4$. As for $I_2(\ve)$, integrating by parts after a change of variables we have
\begin{equation}\notag
\begin{split}
I_2(\ve)&=\frac{4}{|\log\ve|}\left\{\log y\left(\int_{\frac{\ve}{y}}^{\frac{1}{y}}|\log x|^p \frac{x^2+1}{(x^2-1)^2}\, dx\right)\right\}\Bigg|_{y=\ve}^{y=1}\\
&+\frac{4}{|\log\ve|} \int_\ve^1\frac{\log y}{y^2}|\log \frac 1y|^p\frac{\frac{1}{y^2}+1}{\left(\frac{1}{y^2}-1\right)^2}\, dy\\
&-\frac{4\ve}{|\log\ve|} \int_\ve^1\frac{\log y}{y^2}|\log \frac{\ve}{y} |^p\frac{\left(\frac{\ve}{y}\right)^2+1}{\left(\left(\frac{\ve}{y}\right)^2-1\right)^2}\, dy.
\end{split}
\end{equation}

A direct computation for the first term gives
\begin{equation}\notag
\begin{split}
&\frac{4}{|\log\ve|}\left\{\log y\left(\int_{\frac{\ve}{y}}^{\frac{1}{y}}|\log x|^p \frac{x^2+1}{(x^2-1)^2}\, dx\right)\right\}\Bigg|_{y=\ve}^{y=1}\\
&=4\int_1^{\frac{1}{\ve}} |\log x|^p \frac{x^2+1}{(x^2-1)^2}\, dx,
\end{split}
\end{equation}
which converges to 
\begin{equation}\notag
4\int_1^{+\infty} |\log x|^p \frac{x^2+1}{(x^2-1)^2}\, dx,
\end{equation}
as $\ve\to0$. Moreover, since 

\begin{equation}\notag
\int_0^1 \frac{\log y}{y^2}|\log \frac 1y|^p\frac{\frac{1}{y^2}+1}{\left(\frac{1}{y^2}-1\right)^2}\, dy<+\infty
\end{equation}

the second term in the sum converges to $0$ as $\ve\to0$. 

\medskip

After setting $\frac{\ve}{y}=x$, for the last term in the sum we have
\begin{equation}\notag
\begin{split}
&-\frac{4\ve}{|\log\ve|} \int_\ve^1\frac{\log y}{y^2}|\log \frac{\ve}{y} |^p\frac{\left(\frac{\ve}{y}\right)^2+1}{\left(\left(\frac{\ve}{y}\right)^2-1\right)^2}\, dy\\
&=-\frac{4}{|\log\ve|}\int_\ve^1 \log\left(\frac{\ve}{x}\right)|\log x|^p \frac{x^2+1}{(x^2-1)^2}\, dx\\
&=4\int_\ve^1|\log x|^p \frac{x^2+1}{(x^2-1)^2}\, dx-\frac{4}{|\log\ve|}\int_\ve^1 |\log x|^{p+1} \frac{x^2+1}{(x^2-1)^2}\, dx
\end{split}
\end{equation}
which converges to 
  
\begin{equation}\notag
4\int_0^1|\log x|^p \frac{x^2+1}{(x^2-1)^2}\, dx=4\int_1^{+\infty} |\log x|^p \frac{x^2+1}{(x^2-1)^2}\, dx
\end{equation}
as $\ve\to0$. Summing up, we have
 
\begin{equation}\label{eq:prelim}
\lim_{\ve\to0}[u_\ve]_{\wsp}^p=\lim_{\ve\to0}I_2(\ve)=8\int_1^{+\infty}|\log x|^p \frac{x^2+1}{(x^2-1)^2}\, dx.
\end{equation}

Integrating by parts we obtain
 
\begin{equation}\notag
\begin{split}
&\int_1^{+\infty}|\log x|^p \frac{x^2+1}{(x^2-1)^2}\, dx=p\int_1^{+\infty}\frac{|\log x|^{p-1}}{x^2-1}\, dx\\
&=p\int_0^{1}\frac{|\log t|^{p-1}}{1-t^2}\, dt,
\end{split}
\end{equation}
where we set $t=\frac 1x$. Recall now

\begin{equation}\label{eq:geometric intgamma}
\frac{1}{1-x^2}=\sum_{k=0}^\infty x^{2k},\qquad \int_0^1|\log x|^{p-1} x^{2k}\, dx=\frac{\Gamma(p)}{(1+2k)^p},
\end{equation}
 where $\Gamma(\cdot)$ is the Euler Gamma function. Thanks to \eqref{eq:geometric intgamma} we write

\begin{equation}\label{eq:quasi}
\begin{split}
\int_0^1 \frac{|\log t|^{p-1}}{1-t^2}\, dt&=\sum_{k=0}^\infty\int_0^1|\log t|^{p-1} t^{2k}\, dt=\Gamma(p)\sum_{k=0}^\infty \frac{1}{(1+2k)^p},
\end{split}
\end{equation}
proving \eqref{eq:limsemi}. 
\end{proof}

The upper bound for the optimal exponent follows directly from Proposition \ref{const}. 

\begin{prop}\label{mainprop}
Let $sp=1$. There exists $\alpha^*:=\gamma_s^{\frac{s}{1-s}}$ such that
\begin{equation}\notag
\sup_{u\in\wspi,[u]_{\wsp}\leq 1} \int_I e^{\alpha |u|^{\frac{1}{1-s}}}\, dx=+\infty\quad\text{ for }\alpha\in(\alpha^*,+\infty).
\end{equation}
\end{prop}
\begin{proof}
Let $u_\ve$ be the family of functions in $\wspi$ defined in \eqref{eq:moserfunctions}. Thanks to Proposition \ref{const} we have that $[u_\ve]_{\wsp}\to (\gamma_s)^{\frac 1p}$ as $\ve\to0$. Fix $\alpha>\gamma_s^{\frac{s}{1-s}}$. 
For $\ve$ small enough, there exists $\beta>0$ such that $\alpha[u_\ve]^{-\frac{1}{1-s}}\geq\beta>1$. If we set $v_\ve:=\frac{u_\ve}{[u_\ve]}$ we have
\begin{equation}\notag
\int_I e^{\alpha|v_\ve|^{\frac{1}{1-s}}}\, dx\geq \int_{-\ve}^\ve e^{\alpha|v_\ve|^{\frac{1}{1-s}}}\, dx \geq \int_{-\ve}^\ve e^{-\beta\log\ve}\, dx=2\ve^{1-\beta}\to +\infty
\end{equation}
as $\ve\to0$, since $\beta>1$.
\end{proof}

\section{Proof of Theorem \ref{inequalityR}}

We shall adapt a technique by Ruf \cite{Ruf-H1} to our setting.

\medskip

For a measurable function $u$ we set $|u|^{*}:\R{}\to \R{}_+$ to be its non-increasing symmetric rearrangement, whose definition we shall now recall.
For a measurable set $A\subset \mathbb{R}$, we define 
$$A^{*}=(-|A|/2,|A|/2).$$
The set $A^{*}$ is symmetric (with respect to $0$) and  $|A^*|=|A|$. For a non-negative measurable function $f$, such that 
$$|\{x\in \mathbb{R} : f(x)>t\}|<\infty \quad \text{ for every } t>0,$$
we define the symmetric non-increasing rearrangement of $f$ by
$$f^{*}(x)=\int_{0}^{\infty} \chi_{\{y\in\R{} : f(y)>t\}^{*}}(x) dt.$$
Notice that $f^*$ is even, i.e. $f^*(x)=f^*(-x)$ and non-increasing (on $[0,\infty)$).

We will state here the two properties that we shall use in the proof of Proposition \ref{inequalityR}. The following one is proven e.g. in \cite[Section 3.3]{LL}.

\begin{prop}\label{propu*}
Given a measurable function $F:\mathbb{R}\to \mathbb{R}$ and a non-negative non-decreasing function $f:\R{}\to\R{}$, it holds
$$\int_{\mathbb{R}}F(f)dx =\int_{\mathbb{R}}F(f^{*})dx.$$
\end{prop}

The following P\'olya-Szeg\H{o} type inequality can be found e.g. in \cite[Theorem 9.2]{Almgren.Lieb}.

\begin{teo}\label{pol}
Let $0<s<1$ and $u\in \wsp$. Then
$$[|u|^*]_\wsp\leq[u]_\wsp.$$ 
\end{teo}

Now given $u\in \wsp$, from Proposition \ref{propu*} we get
$$\int_{\mathbb{R}}\Phi(\alpha (|u|)^{\frac{1}{1-s}}) \, dx =\int_{\mathbb{R}}\Phi(\alpha (|u|^*)^{\frac{1}{1-s}})\, dx,\quad \| |u|^*\|_{L^p}=\|u\|_{L^p},$$
and according to Theorem \ref{pol}
$$\||u|^*\|_{\wsp}^p=\||u|^*\|_{L^p(\R{})}^p+[|u|^*]_{\wsp}^p\leq \|u\|_{L^p(\R{})}^p+[u]_{\wsp}^p= \|u\|_{\wsp}^p.$$
Therefore in the rest of the proof of \eqref{eq:MTR} we may assume that $u\in \wsp$ is even, non-increasing on $[0,\infty)$, and $\|u\|_{\wsp}\le 1$. We will use a technique by Ruf \cite{Ruf-H1} (see also \cite{IMM}) and write
\begin{equation}\notag
\begin{split}
&\int_{\mathbb{R}}\Phi(\alpha (|u|)^{\frac{1}{1-s}})\, dx \\
=&\int_{I^c}\Phi(\alpha (|u|)^{\frac{1}{1-s}})\, dx+\int_{I}\Phi(\alpha (|u|)^{\frac{1}{1-s}})\,  dx\\
=&:(I)+(II),
\end{split}
\end{equation}
where $I=(-r_0,r_0)$, with $r_0>0$ to be chosen.
Notice that since $u$ is even and non-increasing, for $x\neq0$ and $p>1$, we have
\begin{equation}\label{eq:utile}
|u(x)|^p\leq\frac{1}{2|x|}\int_{-|x|}^{|x|}|u(y)|^p\, dy\leq \frac{\|u\|^{p}_{L^{p}}}{2|x|}.
\end{equation}

We start by bounding $(I)$.  We observe that for $r_0>>1$, we have $|u(x)|\leq 1$ on $I^c$ and hence
$$ |u|^{\frac{p\lceil p-1\rceil}{p-1}}\leq |u|^p\quad \text{on $I^c$,} $$
since $\frac{p\lceil p-1\rceil}{p-1}\geq p$. 
For $k>p-1$ we bound
\begin{equation}\notag
\int_{I^c}(|u|^{p})^{\frac{k}{p-1}}\, dx\leq \int_{I^c} \left(\frac{\|u\|_{L^p}^{p}}{2|x|}\right)^{\frac{k}{p-1}}=\frac{\|u\|_{L^p}^{\frac{pk}{p-1}}r_0^{1-\frac{k}{p-1}}(p-1)}{2^{\frac{k}{p-1}}(k+1-p)}.
\end{equation}

Hence

\begin{equation}\notag
\begin{split}
(I)&=\sum_{k=\lceil p-1\rceil}^{\infty} \int_{I^c} \frac{\alpha^{k}}{k!} |u|^{\frac{kp}{p-1}}\, dx\\
&= \frac{\alpha^{\lceil p-1\rceil}}{\lceil p-1\rceil!}\int_{I^c} |u|^{\frac{p\lceil p-1\rceil}{p-1}}\, dx+\sum_{k=\lceil p\rceil}^{\infty} \int_{I^c} \alpha^{k} \frac{|u|^{\frac{kp}{p-1}}}{k!}dx\\
&\leq C(\alpha,p)\|u\|^p_{L^p}+r_0 (p-1)\sum_{k=\lceil p\rceil}^{\infty}\frac{\alpha^k\left(\|u\|_{L^p}^p\right)^{\frac{k}{p-1}}}{k!(k+1-p)(2r_0)^{\frac{k}{p-1}}}\\
&\leq C(\alpha,p)\|u\|^p_{L^p}+C\sum_{k=\lceil p\rceil}^{\infty}\left(\frac{\alpha}{(2r_0)^{p-1}}\right)^{k}\frac{1}{k!(k+1-p)}\leq C.
\end{split}
\end{equation}
As for $(II)$, define $v\in\wspi$ as follows
\begin{equation}\notag
v(x)=\begin{cases}
u(x)-u(r_0)\quad &|x|\leq r_0\\
0\quad &|x|>r_0.
\end{cases}
\end{equation}
Let $x\in I$. We compute using the monotonicity of $u$ 

\begin{equation}\label{eq:dentro}
\begin{split}
&\int_0^\infty|v(x)-v(y)|^p\frac{x^2+y^2}{(x^2-y^2)^2}\, dy\leq \int_0^\infty|u(x)-u(y)|^p\frac{x^2+y^2}{(x^2-y^2)^2}\, dy.
\end{split}
\end{equation}
Let $x\in I^c$. We have
\begin{equation}\label{eq:fuori}
\begin{split}
&\int_0^\infty |v(x)-v(y)|^p\frac{x^2+y^2}{(x^2-y^2)^2}\, dy\\
&=\int_I |u(r_0)-u(y)|^p\frac{x^2+y^2}{(x^2-y^2)^2}\, dy\\
&\leq \int_I |u(x)-u(y)|^p\frac{x^2+y^2}{(x^2-y^2)^2}\, dy.
\end{split}
\end{equation}
Combining \eqref{eq:dentro}, \eqref{eq:fuori} and integrating in $x$, we get
\begin{equation}\label{eq:seminorma}
[v]^p\leq [u]^p.
\end{equation}

Using the definition of $v$ and the inequality $(a+b)^\sigma\leq a^\sigma+\sigma 2^{\sigma-1}(a^{\sigma-1}b+b^\sigma)$ for $a,b\geq0$ and $\sigma\geq 1$, we have
\begin{equation}\label{eq:allafine}
\begin{split}
u^{\frac{1}{1-s}}&\leq v^{\frac{1}{1-s}}+\frac{1}{1-s}2^{\frac{s}{1-s}}(v^{\frac{s}{1-s}} u(r_0)+ u(r_0)^{\frac{1}{1-s}})\\
&\leq v^{\frac{1}{1-s}}\left(1+\frac{2^{\frac{2s-1}{1-s}}}{p r_0(1-s)}||u||_p^p\right)+2^{\frac{s}{1-s}}+\frac{2^{\frac{s}{1-s}}}{1-s}{r_0}\\
&= v^{\frac{1}{1-s}}\left(1+\frac{2^{\frac{2s-1}{1-s}}}{p r_0(1-s)}||u||_p^p\right)+C(r_0).
\end{split}
\end{equation}
This implies
\begin{equation}\notag
\begin{split}
u(x)&\leq v(x)\left(1+\frac{2^{\frac{2s-1}{1-s}}}{p r_0(1-s)}||u||_p^p\right)^{1-s}+C^{1-s}(r_0)\\
&:=w(x)+C^{1-s}(r_0).
\end{split}\end{equation}

From \eqref{eq:seminorma} and the definition of $w$, we get
\begin{equation}
\begin{split}
[w]^p&=[v]^p\left(1+\frac{2^{\frac{2s-1}{1-s}}}{p r_0(1-s)}||u||_p^p\right)^{\frac{1-s}{s}}\\
&\leq\left(1-||u||_p^p\right)\left(1+\frac{2^{\frac{2s-1}{1-s}}}{p r_0(1-s)}||u||_p^p\right)^{\frac{1-s}{s}}
\end{split}
\end{equation}

Consider now the function $f(t)=(1-t)(1+\tau t)^\sigma$, where $\tau:=\frac{2^{\frac{2s-1}{1-s}}}{p r_0(1-s)}$ and $\sigma=\frac{1-s}{s}>0$. 
We compute 
\begin{equation}
f'(t)=(1+\tau t)^{\sigma-1}\left(\tau t(-\sigma-1)+\tau\sigma-1 \right)
\end{equation}
which vanishes for $t_1=-\frac 1\tau<0$ and $t_2=\frac{\tau\sigma-1}{\tau(\sigma+1)}$. We choose now $r_0>2^{\frac{2s-1}{1-s}}$ so that $t_2<0$. This implies that $f$ is decreasing in $(0,1)$ and since $f(0)=1$ we have that $f(t)<1$ for $t\in(0,1)$, which implies
\begin{equation}\label{eq:okok}
[w]^p\leq 1.
\end{equation} 
We can apply now Proposition \ref{inequality} on the interval $I=(-r_0,r_0)$ to get that there exists $\alpha_*>0$ such that
\begin{equation}
\int_I e^{\alpha_* w^{p'}}\, dx\leq C
\end{equation}
and using \eqref{eq:allafine} we get

\begin{equation}
\int_I e^{\alpha_* u^{\frac{1}{1-s}}}\, dx\leq C\int_I e^{\alpha_* w^{\frac{1}{1-s}}}\, dx\leq C,
\end{equation}
concluding the proof of \eqref{eq:MTR}.

\medskip

To prove the second part of the claim one can argue as in the previous section, using the sequence of functions $u_\ve$ defined in \eqref{eq:moserfunctions} and taking into account that now the norm we are working with is the full $W^{s,p}$-norm. Indeed we have 
\begin{equation}\label{veloce}
\|u_\ve\|^p_{L^p}=\int_\R |u_\ve|^p\, dx=\int_{|x|\leq\ve}\left(|\log\ve|^{p-sp}\right)\, dx+\int_{\ve<|x|<1}\frac{|\log x|}{|\log\ve|^{sp}}\, dx=O(|\log\ve|^{-1}).
\end{equation}
Hence from \eqref{eq:limsemi}, it follows that 
\begin{equation}\label{eq:fullnorm}
\lim_{\ve\to0}\|u_\ve\|_{\wsp}^p=\gamma_s.
\end{equation}
Choose $M>0$ large enough so that
$$ \Phi(t)\geq \frac 12 e^t,\quad t\geq M. $$
Then one has
\begin{equation}\label{ultima}
\begin{split}
\int_\R \Phi\left(\gamma_s^s \frac{u_\ve}{\|u_\ve\|_{\wsp}}^{\frac{1}{1-s}}\right)\, dx &\geq \int_{u_\ve\geq M}\Phi\left(\gamma_s^s \frac{u_\ve}{\|u_\ve\|_{\wsp}}^{\frac{1}{1-s}}\right)\, dx\\
&\geq \frac 12\int_{-\ve}^{\ve}e^{\left(\gamma_s^s\frac{u_\ve}{\|u_\ve\|_{\wsp}}\right)^{\frac{1}{1-s}}}\, dx.
\end{split}
\end{equation}
for $\ve$ small enough. Now, thanks to \eqref{eq:fullnorm}, one can argue as in the proof of Proposition \ref{mainprop} to conclude the proof of Theorem \ref{inequalityR}.

\section{Proof of Theorem \ref{sharpinterval}}

We will start by proving \eqref{eq:MTf} since the proof of \eqref{eq:MTfr} will follow adapting the reasoning of the previous section.

\medskip

Let $u_\ve$ be as in \eqref{eq:moserfunctions}.
To prove \eqref{eq:MTf} it is enough to show that there exists a constant $\delta>0$ such that
\begin{equation}\notag
\int_{-\ve}^{\ve} e^{\alpha^*\left(\frac{u_\ve}{[u_\ve]}\right)^{\frac{1}{1-s}}}\, dx\geq\delta.
\end{equation}
Indeed, $u_\ve\to+\infty$ uniformly for $|x|<\ve$ as $\ve\to0$ and we have
\begin{equation}\notag
\sup_{u\in\wspi,[u]_{\wsp}\leq1}\int_I f(|u|)e^{\alpha^* \left(\frac{|u|}{[u]}\right)^{\frac{1}{1-s}}}\, dx\geq \inf_{|x|<\ve}f(|u_\ve|) \int_{-\ve}^{\ve}e^{\alpha^* \left(\frac{|u_\ve|}{[u_\ve]}\right)^{\frac{1}{1-s}}}\, dx.
\end{equation}
From Proposition \ref{const}, it follows that 
\begin{equation}
\lim_{\ve\to0}\frac{[u_\ve]}{\gamma_s^{s}}=1
\end{equation}
and in particular
\begin{equation}\notag
\lim_{\ve\to0}[u_\ve]^p=8\int_1^{+\infty}|\log x|^p\frac{x^2+1}{(x^2-1)^2}\, dx=\gamma_s.
\end{equation}
We compute
\begin{equation}\label{veloce2}
\lim_{\ve\to0}\log\frac{1}{\ve}\left([u_\ve]^p-\gamma_s\right)=8\lim_{\ve\to0}\log\frac{1}{\ve}\int_{\frac{1}{\ve}}^{+\infty}|\log x|^p\frac{x^2+1}{(x^2-1)^2}\, dx=0. 
\end{equation}
Then we can write
\begin{equation}\label{eq:buona}
\frac{[u_\ve]^p}{\gamma_s}\leq1+(C\log\frac{1}{\ve})^{-1}
\end{equation}
and in particular, recalling
\begin{equation}\notag
\lim_{t\to+\infty} \frac{t}{(1+\frac Ct)^{\frac{1}{1-s}}}-t=-\frac{1}{1-s},
\end{equation}
we have
\begin{equation}\label{39}
\begin{split}
&\int_{-\ve}^{\ve}e^{\gamma_s^{\frac{s}{1-s}} \left(\frac{|u_\ve|}{[u_\ve]}\right)^{\frac{1}{1-s}}}\, dx=\int_{-\ve}^{\ve}e^{\left(\frac{\gamma_s^{s}}{[u_{\ve}]}\right)^{\frac{1}{1-s}} |u_\ve|^{\frac{1}{1-s}}}\, dx\\
&\geq\int_{-\ve}^{\ve} e^{\frac{\log\frac{1}{\ve}}{(1+C(\log\frac{1}{\ve})^{-1})^{\frac{1}{1-s}}}}\, dx\\
&=2\ve e^{\frac{\log\frac{1}{\ve}}{(1+C(\log\frac{1}{\ve})^{-1})^{\frac{1}{1-s}}}}\to e^{-\frac{1}{1-s}}
\end{split}
\end{equation}
as $\ve\to0$.
Therefore
\begin{equation}\label{39 2}
\int_Ie^{\gamma_s^{\frac{s}{1-s}} \left(\frac{|u_\ve|}{[u_\ve]}\right)^{\frac{1}{1-s}}}\, dx\geq \delta
\end{equation}
for some $\delta>0$, proving \eqref{eq:MTf}. We shall now prove \eqref{eq:MTfr}. From \eqref{veloce} and \eqref{veloce2} it follows that 
\begin{equation}\label{ueps}
\frac{\|u_\ve\|_{\wsp}^p}{\gamma_s}\leq 1+O(|\log\ve|^{-1}).
\end{equation}

Now using \eqref{ultima} and arguing as in \eqref{39} and \eqref{39 2}, we conclude the proof.


\bibliographystyle{plain}
\bibliography{bib2}

\begin{thebibliography}{10}

\bibitem{adams}
D.~R. Adams.
\newblock A sharp inequality of {J}. {M}oser for higher order derivatives.
\newblock {\em Ann. of Math. (2)}, 128(2):385--398, 1988.

\bibitem{Almgren.Lieb}
Frederick~J. Almgren, Jr. and Elliott~H. Lieb.
\newblock Symmetric decreasing rearrangement is sometimes continuous.
\newblock {\em J. Amer. Math. Soc.}, 2(4):683--773, 1989.

\bibitem{Brasco-Lindgren-Parini}
L.~Brasco, E.~Lindgren, and E.~Parini.
\newblock The fractional {C}heeger problem.
\newblock {\em Interfaces Free Bound.}, 16(3):419--458, 2014.

\bibitem{hyder}
A.~Hyder.
\newblock Moser functions and fractional moser-trudinger type inequalities.
\newblock {\em Nonlinear Anal.}, 146:185--210, 2016.

\bibitem{IMM}
S.~Iula, A.~Maalaoui, and L.~Martinazzi.
\newblock A fractional {M}oser-{T}rudinger type inequality in one dimension and
  its critical points.
\newblock {\em Differential Integral Equations}, 29(5-6):455--492, 2016.

\bibitem{Yud}
V.~I. Judovi{\v{c}}.
\newblock Some estimates connected with integral operators and with solutions
  of elliptic equations.
\newblock {\em Dokl. Akad. Nauk SSSR}, 138:805--808, 1961.

\bibitem{LL}
Elliott~H. Lieb and M.~Loss.
\newblock {\em Analysis}, volume~14 of {\em Graduate Studies in Mathematics}.
\newblock American Mathematical Society, Providence, RI, second edition, 2001.

\bibitem{marfrac}
L.~Martinazzi.
\newblock Fractional {A}dams-{M}oser-{T}rudinger type inequalities.
\newblock {\em Nonlinear Anal.}, 127:263--278, 2015.

\bibitem{Mos}
J.~Moser.
\newblock A sharp form of an inequality by {N}. {T}rudinger.
\newblock {\em Indiana Univ. Math. J.}, 20:1077--1092, 1970/71.

\bibitem{Par-Ruf}
E.~Parini and B.~Ruf.
\newblock {O}n the {M}oser-{T}rudinger inequality in fractional
  {S}obolev-{S}lobodeckij spaces.
\newblock {\em Preprint}, 2015.

\bibitem{Peetre}
J.~Peetre.
\newblock Espaces d'interpolation et th\'eor\`eme de {S}oboleff.
\newblock {\em Ann. Inst. Fourier (Grenoble)}, 16(fasc. 1):279--317, 1966.

\bibitem{Poho}
S.~I. Pohozaev.
\newblock The {S}obolev embedding in the case $pl=n$.
\newblock {\em Proc. Tech. Sci. Conf. on Adv. Sci. Research 1964-1965,
  Mathematics Section, Moskov. Energet. Insti. Moscow}, 16:158--170, 1965.

\bibitem{Ruf-H1}
B.~Ruf.
\newblock A sharp {T}rudinger-{M}oser type inequality for unbounded domains in
  {$\Bbb R^2$}.
\newblock {\em J. Funct. Anal.}, 219(2):340--367, 2005.

\bibitem{Tru}
N.~S. Trudinger.
\newblock On imbeddings into {O}rlicz spaces and some applications.
\newblock {\em J. Math. Mech.}, 17:473--483, 1967.

\end{thebibliography}

\end{document}